\documentclass[12pt,a4paper]{article}
\usepackage[utf8]{inputenc}
\usepackage{amsmath}
\usepackage{amsfonts}
\usepackage{amssymb}
\usepackage{amsthm}
\usepackage{hyperref}
\usepackage[left=2cm,right=2cm,top=2cm,bottom=2cm]{geometry}
\usepackage{xcolor}
\usepackage[all,cmtip]{xy}
\usepackage{tikz}
\usepackage{enumitem}

\newcommand{\U}{\mathrm{U}}
\newcommand{\SU}{\mathrm{SU}}

\newcommand{\Hom}{\mathrm{Hom}}

\newcommand{\mft}{\mathfrak{t}}

\newcommand{\RR}{\mathbb{R}}
\newcommand{\CC}{\mathbb{C}}
\newcommand{\ZZ}{\mathbb{Z}}

\newtheorem{thm}{Theorem}[section]
\newtheorem{prop}[thm]{Proposition}
\newtheorem{cor}[thm]{Corollary}
\newtheorem{lem}[thm]{Lemma}

\theoremstyle{definition}
\newtheorem{rem}[thm]{Remark}

\begin{document}

\title{GKM theory and Hamiltonian non-K\"ahler actions in dimension $6$}
\author{Oliver Goertsches\footnote{Philipps-Universit\"at Marburg, email: goertsch@mathematik.uni-marburg.de}, Panagiotis Konstantis\footnote{Universität zu Köln, email: pako@math.uni-koeln.de}, and Leopold Zoller\footnote{Ludwig-Maximilians Universität München, email:
Leopold.Zoller@mathematik.uni-muenchen.de}}

\maketitle
\begin{abstract}
Using the classification of $6$-dimensional manifolds by Wall, Jupp and \v Zubr, we observe that the diffeomorphism type of simply-connected, compact $6$-dimensional integer GKM $T^2$-manifolds is encoded in their GKM graph. As an application, we show that the $6$-dimensional manifolds on which Tolman and Woodward constructed Hamiltonian, non-K\"ahler $T^2$-actions with finite fixed point set are diffeomorphic to Eschenburg's twisted flag manifold $\SU(3)//T^2$. In particular, they admit a noninvariant K\"ahler structure.
\end{abstract}

\section{Introduction}

The first example of a compact Hamiltonian torus action with finite fixed point set not admitting an invariant K\"ahler structure was given by Tolman \cite{Tolman} in the mid 90s. Builing upon her work, Woodward \cite{Woodward} produced a similar example that even extends to a multiplicity-free Hamiltonian action. While Tolman used symplectic gluing of two $6$-dimensional Hamiltonian $T^2$-manifolds that are restrictions of toric symplectic manifolds, Woodward's example is a $\U(2)$-equivariant symplectic surgery of the $6$-dimensional full flag manifold $\U(3)/T^3$.

Until now it was unknown if these examples admit any K\"ahler structure at all, see \cite[Remark, p.\ 309]{Tolman}, as well as if these manifolds are in fact the same; Woodward conjectures them to be equivariantly diffeomorphic, see \cite[Footnote 3]{Woodward}.

In this paper, we compare these examples to a third closely related example, namely a Hamiltonian
$T^2$-action on Eschenburg's twisted flag $\SU(3)//T^2$, constructed in \cite{1812.09689v1}. We
will show that Tolman's and Woodward's examples are (nonequivariantly) diffeomorphic to this
manifold; as it is known \cite[Theorem 2]{Eschenburg2}, \cite{EscherZiller}, \cite[Section
4]{1812.09689v1} that the Eschenburg flag admits a K\"ahler structure we can answer also the
question on the existence of a (noninvariant) K\"ahler structure on these examples in the
affirmative. More precisely, we show in
Theorem \ref{thm:Kahler} that the $T^2$-invariant symplectic form on the Eschenburg flag does
itself admit a compatible complex structure. However, by the work of Tolman such a complex
structure can never be $T^2$-invariant. To our knowledge, this is the first example of
a Hamiltonian action with this property.

The main tools to derive our conclusions are the diffeomorphism classification of $6$-dimensio\-nal manifolds by Wall \cite{MR0215313}, Jupp \cite{MR0314074} and \v Zubr \cite{MR970082} and integer equivariant cohomology, most importantly GKM theory. We show that for GKM manifolds all necessary topological invariants, i.e., the integer cohomology ring as well as the first Pontrjagin class and the second Stiefel-Whitney class, are encoded in the GKM graph of the action; for Hamiltonian GKM actions with connected isotropy groups, in the so-called x-ray.
\\

\noindent {\bf Acknowledgements.} We would like to thank to Maximilian Schmitt for drawing our attention to Tolman's and Woodward's examples. We also want to express our gratitude towards Nicholas Lindsay and Dmitri Panov for helpful comments which led to the formulation and proof of Theorem \ref{thm:Kahler}.
The third named author is supported by the German
Academic Scholarship foundation.

\section{Preliminaries}

Throughout this paper we consider actions of compact tori $T$ on closed,  connected manifolds $M$, as well as their equivariant cohomology $H^*_T(M,\ZZ)$ with integer coefficients. This is the cohomology of the Borel construction
\[
ET\times_T M,
\]
equipped with the $H^*(BT,\ZZ)$-algebra structure induced by the natural projection $ET \times_T M \to BT$. Note that for a $k$-dimensional torus $T$ we have $H^*(BT,\ZZ)\cong \ZZ[u_1,\ldots,u_k]$, where the $u_i$ are the transgressions of the generators of $H^*(T,\ZZ)$ in the fibration $T\to ET\to BT$, see \cite[p.\ 410f]{Borel}.

For a $T$-action on $M$, we denote by $M^T$ the fixed point set, and by $M_1=\{p\in M\mid \dim T\cdot p \leq 1\}$ the \emph{one-skeleton} of the action. In case that $M$ is orientable, we say that the action satisfies the \emph{GKM conditions} (named after Goresky--Kottwitz--MacPherson \cite{GKM}) if $M^T$ is a finite set of points, and $M_1$ is a finite union of $T$-invariant two-spheres. Note that many papers include the vanishing of the odd degree cohomology into the GKM conditions, but as we consider integer instead of the more common real coefficients, we will list this condition separately in our results.

We denote by $\ZZ_\mft^*\subset \mft^*$ the weight lattice of $T$. The GKM condition implies that at any fixed point $p$ the isotropy representation decomposes into two-dimensional irreducible summands, with weights $\alpha_1,\ldots,\alpha_n$, where the dimension of $M$ is $2n$. These weights are elements in $\ZZ_\mft^*/\pm 1$.

To an action of a torus $T$ on $M$ satisfying the GKM conditions one associates the \emph{GKM graph} $\Gamma$ as follows. Its vertex set $V(\Gamma)$ contains one vertex for each fixed point, and its set of (unoriented) edges $E(\Gamma)$ contains one edge for every invariant two-sphere $S$, connecting the two vertices corresponding to the two fixed points in $S$. In other words, the GKM graph is the space $M_1/T$, considered as a graph. Additionally we label every edge $e\in E(\Gamma)$ with the weight $\alpha_e$ of the corresponding two-sphere.

If we are given a $T$-invariant almost complex or symplectic structure on $M$, then the weights are well-defined elements of $\ZZ_\mft^*$ and we consider a \emph{signed} variant of the \emph{GKM graph}. The underlying graph is the same, but the label of an edge now associates a unique sign to each possible orientation of the edge in the following way: for an oriented edge $e$ we denote the initial vertex by $i(e)$ and the terminal vertex by $t(e)$. The weight $\alpha_e$ is then by definition the weight of the isotropy representation of the action on the two-sphere at the point $i(e)$. Note that the weight of the same two-sphere at $t(e)$ is then $-\alpha_e$. This is the same as $\alpha_{\bar{e}}$, where $\bar{e}$ denotes the edge $e$, but inverted.

If we have GKM actions of a torus $T$ on two closed, connected manifold $M$ and $N$, then an \emph{isomorphism} of the GKM graphs $\Gamma_M$ and $\Gamma_N$ of $M$ and $N$ consists of an isomorphism $\varphi\colon \Gamma_M\to \Gamma_N$ of abstract graphs, together with an automorphism $\psi\colon T\to T$ which intertwines the labels, i.e., $\alpha_{\varphi(e)} = \alpha_{e}\circ (d\psi)^*\in \ZZ_\mft^*/\pm 1$. If we are given invariant almost complex structures on $M$ and $N$, then we ask an isomorphism of the signed GKM graphs to respect the signed labels. This notion of isomorphism of signed GKM graphs is the same as that in \cite[Definition 3.1]{GuilleminSabatiniZara}, but note that other natural notions are possible. For example, the notion of isomorphism used in \cite{FranzYamanaka} does not include an automorphism of $T$.

\section{Topological invariants and diffeomorphism type via GKM theory}\label{sec:topinvgkm}
The goal of this section is to show the following theorem, which states that in dimension $6$, the diffeomorphism type of a GKM manifold is determined by its graph.

\begin{thm}\label{thm:gkm6dim}
 Let $M$ and $N$ be compact, orientable, connected, smooth manifolds satisfying
  $H^{odd}(M,\ZZ)=H^{odd}(N,\ZZ)=0$. Consider actions of a torus $T$ on $M$ and $N$ satisfying the GKM conditions, such that for
  all $p\notin M_1$, the isotropy group $T_p$ is contained in a proper subtorus of $T$, and analogously for $N$. Let further $\varphi\colon\Gamma_M\to \Gamma_N$ be an isomorphism of GKM graphs. Then:
 \begin{enumerate}[label=(\alph*)]
   \item The isomorphism $\varphi$ induces an isomorphism $H^\ast(N,\ZZ)\to H^\ast(M,\ZZ)$
      which maps the Pontrjagin classes of $N$ to those of $M$, and is such that the induced isomorphism $H^\ast(N,\ZZ_2)\cong H^\ast(N,\ZZ)\otimes \ZZ_2\to H^\ast(M,\ZZ)\otimes \ZZ_2\cong H^\ast(M,\ZZ_2)$ maps the Stiefel-Whitney classes of $N$ to those of $M$.
   \item If $M$ and $N$ are additionally $6$-dimensional and simply-connected, then the isomorphism from (a) is induced by a (nonequivariant) diffeomorphism $M\to N$.
   \item Assume, in the situation of (a), that $M$ and $N$ are equipped with $T$-invariant almost complex structures, and that $\varphi$ is an isomorphism of signed GKM graphs. Then the isomorphism $H^*(N,\ZZ)\to H^*(M,\ZZ)$ also maps the Chern classes of $N$ to the Chern classes of $M$. With the additional assumptions from (b), the diffeomorphism $M\to N$ then also respects the homotopy class of the almost complex structures.
 \end{enumerate}
\end{thm}
The assumption on the isotropy groups is obviously satisfied if all isotropy groups are connected. By \cite[Lemma 2.1]{MR2283418}, for a $T$-action on a closed, compact, orientable manifold $M$ with finite fixed point set, the vanishing of $H^{odd}(M,\ZZ)$ is equivalent to the freeness of the $H^*(BT,\ZZ)$-module $H^*_T(M,\ZZ)$. Also, by \cite[Theorem 5.1]{MR2308029}, the equivariant cohomology $H^*_T(M,\ZZ)$ of a compact Hamiltonian $T$-manifold with connected isotropy groups, such that the cohomology of the fixed point set is torsion-free, is a free module. This directly implies the following corollary of Theorem \ref{thm:gkm6dim}:

\begin{cor}\label{cor:mainthm}
Given a Hamiltonian $T^2$-action with connected isotropy groups and satisfying the GKM conditions on a simply-connected compact $6$-dimensional symplectic manifold $M$, the (nonequivariant) diffeomorphism type of $M$ and the homotopy class of a compatible almost complex structure are encoded in the signed GKM graph of the action.
\end{cor}

To understand the relation to part (c) in the Theorem \ref{thm:gkm6dim} we remind the reader that it is possible to define Chern classes for symplectic manifolds
$(M,\omega)$. Indeed, if $J$ is a compatible almost complex structure, then one defines the total
Chern class $c(M,J)$ as the total Chern class of the complex vector bundle $(TM,J)$. The class $c(M,J)$ is
independent of the choice of $J$: if $J'$ is another almost complex structure adapted to the
symplectic form $\omega$, then consider $g$ and $g'$ Riemannian metrics associated to
$(\omega,J)$ and
$(\omega,J')$ respectively. The metrics $g_t := (1-t)g + tg'$ for $t\in [0,1]$ define a path
of compatible almost complex structures $J_t$ with respect to $\omega$ such that $J_0 =J$ and $J_1 = J'$. Thus $J$ and $J'$ are
homotopic and therefore $c(M,J) = c(M,J')$.

\begin{rem}\label{rem:xraygkm} For a Hamiltonian action of a torus $T$ on a compact symplectic manifold $M$, with momentum map $\mu\colon M\to \mft^*$, Tolman \cite{Tolman} and Woodward \cite{Woodward} considered the so-called \emph{x-ray} of the action, which is defined as follows: For any isotropy subgroup $H=T_p$ of the action, the image $\mu(N)$ of a component $N$ of the fixed point set $M^H$ of $H$ is convex polytope, contained in the convex polytope $\mu(M)$; it is the convex hull of the images of those $T$-fixed points that are contained in $N$. The x-ray is the collection of all these polytopes, for all isotropy subgroups $H$ and components $N$. (Tolman also includes the orbit type stratification of $M$ into the definition of the x-ray.) In Section \ref{sec:TWE} we will encounter several examples of x-rays.

If the Hamiltonian $T$-action additionally satisfies the GKM conditions, then the GKM graph can almost be read off from the x-ray: the vertices and edges are precisely given by the zero- and one-dimensional polytopes in the x-ray. Any one-dimensional polytope corresponds precisely to a two-sphere in the one-skeleton of the action, hence to an edge in the GKM graph. The slope of the polytope is a positive multiple of the corresponding weight of the isotropy representation. The x-ray therefore determines the labels of the GKM graph up to a positive multiple. If all the isotropy groups of points in the one-skeleton are connected, then the weights are primitive elements of $\ZZ_{\mft}^*$, hence uniquely determined by the x-ray. This situation will occur in Tolman's and Woodward's examples, in Section \ref{sec:TWE} below.
\end{rem}

We start the proof of Theorem \ref{thm:gkm6dim} with the dimension-independent considerations. In the proof of \cite[Lemma 2.1]{MR2283418} it is shown that for a smooth compact $T$-manifold with $H^{odd}(M,\ZZ)=0$ and finite fixed point set the natural map
\[
H^*_T(M,\ZZ) \longrightarrow H^*(M,\ZZ)
\]
is surjective, with kernel equal to the ideal generated by the image of $H^{>0}(BT,\ZZ)$ in $H^*_T(M,\ZZ)$. In particular, in this situation the ordinary cohomology ring $H^*(M,\ZZ)$ is determined by the equivariant cohomology ring. If moreover for all $p\notin M_1$, the isotropy group $T_p$ is contained in a proper subtorus of $T$, then it is shown in \cite[Corollary 2.2]{FranzPuppe} that the Chang-Skjelbred lemma holds true, i.e., that there is an exact sequence
\begin{equation}\label{eq:chang}
0 \longrightarrow H^*_T(M,\ZZ) \longrightarrow H^*_T(M^T,\ZZ)\longrightarrow H^*_T(M_1,M^T,\ZZ).
\end{equation}
In other words, the image of the restriction map $H^*_T(M,\ZZ)\longrightarrow H^*_T(M^T,\ZZ)$ equals that of the restriction map $H^*_T(M_1,\ZZ)\longrightarrow H^*_T(M^T,\ZZ)$. Now, given that the GKM conditions hold true, the $T$-space $M_1$ is determined entirely by the GKM graph of the action: $M_1$ is a finite union of two-spheres, acted on by $T$ in a way determined by the corresponding labels, joined together at the fixed points as prescribed by the graph. This shows that in the situation of Theorem \ref{thm:gkm6dim} the GKM graph determines the equivariant, and hence the ordinary cohomology: an isomorphism of GKM graphs $\varphi\colon \Gamma_M\to \Gamma_N$ of two GKM manifolds $M$ and $N$ induces a homeomorphism of one-skeleta $\varphi_1\colon M_1\to N_1$, which is twisted equivariant with respect to the automorphism $\psi\colon T\to T$. This homeomorphism, restricted to the respective set of fixed points, defines an isomorphism of equivariant cohomologies $H^*_T(N^T,\ZZ)\to H^*_T(M^T,\ZZ)$ which is twisted $H^*(BT,\ZZ)$-linear with respect to $\psi$. It restricts to an isomorphism of the images of the equivariant cohomologies of the one-skeleta, and thus we obtain a twisted linear isomorphism of equivariant cohomologies $H^*_T(N,\ZZ)\to H^*_T(M,\ZZ)$. Dividing by the ideals generated by the image of $H^{>0}(BT,\ZZ)$ we obtain an isomorphism $H^*(N,\ZZ)\to H^*(M,\ZZ)$.

We now (still in arbitrary dimension) show that the Pontrjagin and Stiefel-Whitney classes of $M$
are encoded in the GKM graph. To do this, we need to consider the equivariant versions of these
characteristic classes defined through the Borel model,  c.f. \cite{MR1850463}. Consider
a characteristic class $c$, where we assume that $c$ lies in a cohomology group
with coefficients in a ring $R$. Suppose now $\pi \colon E \to M$ is a $T$-equivariant\footnote{Of course this can be conducted for any compact Lie group.} vector
bundle. Then
\begin{equation}\label{eq:borelbundle}
\pi_T := \textrm{id} \times \pi \colon ET \times_T E \to ET \times_T M
\end{equation}
defines a new $T$-equivariant vector bundle over the homotopy quotient $ET \times_T M$. Thus it is possible to consider the \emph{equivariant
class}
\[
  c_T(E) := c(ET \times_T E) \in H^\ast(ET\times_T M, R) = H^\ast_T(M, R).
\]
In this way we obtain the \emph{(integral) equivariant Pontrjagin class} or the \emph{equivariant
Stiefel-Whitney class}. If $\pi\colon E \to M$ is a complex vector bundle and the $T$-action
preserves the complex structure then $\pi_T$ is again a complex vector bundle and in this way we obtain the \emph{(integral) equivariant Chern class}, cf.\ \cite[pp. 290]{MR1150492}.

\begin{prop}
Let $M$ and $N$ be GKM manifolds. Then for any isomorphism between the GKM graphs of $M$ and $N$, the induced isomorphism $H^*(N,\ZZ)\to H^*(M,\ZZ)$ sends the Pontrjagin classes of $N$ to those of $M$. After tensoring with $\ZZ_2$, the same holds for the Stiefel-Whitney classes. If there are invariant almost complex structures on $M$ and $N$, and $\varphi$ is an isomorphism of signed GKM graphs, the same holds for the Chern classes.
\end{prop}
\begin{proof}
Upon restriction to the fixed point sets $M^T$ and $N^T$, the tangent bundles become $T$-equivariant vector bundles over discrete point sets, which are uniquely determined by the weights of the isotropy representation at the fixed points, i.e., by data contained in the GKM graph. Thus, pulling back $TN|_{N^T}$ with ${\varphi_1}|_{M^T}$ and composing the action with $\psi$ we obtain the $T$-vector bundle $TM|_{M^T}$. Hence the twisted linear isomorphism $H^*_T(N^T,\ZZ)\to H^*_T(M^T,\ZZ)$ intertwines the equivariant Pontrjagin classes of these bundles, and, after tensoring with $\ZZ_2$, also the equivariant Stiefel-Whitney classes. The same holds for the equivariant Chern classes in the presence of an invariant almost complex structure.

By naturality, upon restriction to $M^T$ and $N^T$, the equivariant characteristic classes of $M$ and $N$ are mapped to the corresponding characteristic classes of the bundles described above. As in our situation $H^*_T(M,\ZZ)$ is a free $H^*(BT,\ZZ)$-module (and after tensoring with $\ZZ_2$, also $H^*_T(M,\ZZ_2)$ a free $H^*(BT,\ZZ_2)$-module), the Borel localization theorem (see e.g.\ \cite[Theorem (3.2.6)]{AlldayPuppe}) implies that the inclusion map $M^T\to M$ induce injective homomorphisms in equivariant cohomology, both with integer and $\ZZ_2$ coefficients. The same holds for $N$. This shows that the isomorphism $H^*_T(N,\ZZ)\to H^*_T(M,\ZZ)$ intertwines the equivariant characteristic classes of $N$ and $M$.

The isomorphism $H^*(N,\ZZ)\to H^*(M,\ZZ)$ defined above was constructed by dividing by the ideals $(H^{>0}(BT,\ZZ)\cdot H^*_T(N,\ZZ))$ and $(H^{>0}(BT,\ZZ)\cdot H^*_T(M,\ZZ))$. The natural projection
\[
H^*_T(M,\ZZ)\to H^*_T(M,\ZZ)/(H^{>0}(BT,\ZZ)\cdot H^*_T(M,\ZZ))\cong H^*(M,\ZZ)
\] is nothing but the map induced by the fiber inclusion of the fibration $M\to ET\times_T M \to BT$. Pulling back the vector bundle $ET\times_T TM \to ET \times_T M$ via this fiber inclusion we get back the original bundle $TM \to M$. This implies by naturality that the equivariant characteristic classes are mapped to the ordinary characteristic classes, and hence the isomorphism $H^*(N,\ZZ)\to H^*(M,\ZZ)$ intertwines the ordinary characteristic classes.
\end{proof}

This completes the proof of part (a) and the first statement in part (c) of
Theorem \ref{thm:gkm6dim}. Although it is not necessary for the proof of the main theorem we want
to complement the above discussion with explicit formulas for the computation of the
characteristic classes from the GKM graph. \begin{prop}\label{prop:Formeln} For a $T$-manifold
  $M$ with finite fixed point set, the restriction of the total equivariant Pontrjagin class of
  $M$ to $M^T$ is given by
\[
\sum_{i=1}^n \prod_{j=1}^m (1+\alpha_{ij}^2)\in \bigoplus_{i=1}^n H^*(BT,\ZZ),
\]
where $M^T=\{p_1,\ldots,p_n\}$, and $\pm\alpha_{ij}\in H^2(BT,\ZZ)$ are the weights (up to sign) of the isotropy representation of $T$ on $T_{p_i}M$. The restriction of the total equivariant Stiefel-Whitney class of $M$ to $M^T$ is given by
\[
\sum_{i=1}^n\prod_{j=1}^m (1\pm\alpha_{ij}) \in \bigoplus_{i=1}^n H^*(BT,\ZZ_2).
\]
Moreover, if there is an invariant almost complex structure on $M$, the signs of the $\alpha_{ij}$ are uniquely determined and the restriction of the total Chern class equals
\[
\sum_{i=1}^n\prod_{j=1}^m (1+\alpha_{ij})\in \bigoplus_{i=1}^n H^*(BT,\ZZ).
\]
\end{prop}

\begin{proof} All three expressions follow from the naturality of characteristic classes which reduces the problem to calculating the characteristic classes of the equivariant bundles $T_{p_i}M\to \{p_i\}$ over the fixed points. The expression for the total equivariant Pontrjagin class was given in \cite{MR3456711}, see the end of Section $2$ therein. It follows because the isotropy representation $T_{p_i}M$ admits an invariant almost complex structure; by \cite[Lemma 6.10]{MR1150492} its total equivariant Chern class is given by $\prod_{j=1}^m (1 \pm \alpha_{ij})$, and \cite[Corollary 15.5]{MR0440554} implies that the Pontrjagin classes are determined by the Chern classes in the above manner.

In the presence of an invariant almost complex structure on $M$, the choice of almost complex structure on the $T_{p_i}M$ is canonical and we obtain the expression for the Chern classes with unique signs.

Finally, the statement for the total equivariant Stiefel-Whitney class uses the same argument, combined with the fact that the mod $2$ reduction of the total Chern class of a complex vector bundle is the total Stiefel-Whitney class, see \cite[Problem 14-B, p.\ 171]{MR0440554}. This was, in the context of quasitoric manifolds, also used in \cite[Corollary 6.7]{MR1104531}.
\end{proof}

We now specialize to the $6$-dimensional setting. To finish the proof of Theorem \ref{thm:gkm6dim}, we need to invoke the diffeomorphism classification of simply-connected smooth $6$-dimensional manifolds \cite{MR0314074} which we now summarize, specialized to the case of manifolds with vanishing odd-degree cohomology.

Let $M$ be an oriented, closed and simply-connected smooth $6$-manifold with $H^{odd}(M,\ZZ)=0$.  The latter condition implies that $M$ has torsion-free homology using Poincar\'{e} duality and the universal coefficient theorem. Let us consider the following invariants of $M$:
\begin{itemize}
  \item $H:=H^2(M;\ZZ)$ a finitely generated free abelian group,
  \item $\mu_M \colon H \otimes H \otimes H \to \ZZ$, a symmetric homomorphism defined by
    \[
      \mu_M(x\otimes y\otimes z)= \langle x\cup y\cup z, [M] \rangle
    \]
    where $[M]$ is the fundamental class of $M$ and $\langle \cdot , \cdot  \rangle$ the
    Kronecker-pairing,
  \item $w_2(M) \in H^2(M;\ZZ_2) \cong H\otimes \ZZ_2$ the second Stiefel-Whitney class (this isomorphism is induced by the homomorphism of coefficients $\ZZ\to \ZZ_2$),
  \item $p_1(M) \in H^4(M;\ZZ)\cong \Hom_{\ZZ}(H,\ZZ)$ the first Pontrjagin class.
\end{itemize}

We call a quadruple $(H,\mu,w,p)$ a \emph{system of invariants} if $H$ is a finitely generated
free abelian group, $\mu\colon H\otimes H\otimes H \to \ZZ$ a symmetric homomorphism, $w$
an element of $H\otimes \ZZ_2$ and $p \in \Hom_{\ZZ}(H,\ZZ)$.
Two systems of invariants $(H,\mu,w,p)$ and
$(H',\mu',w',p')$ are \emph{equivalent} if there is an isomorphism $\Phi \colon H \to H'$
such that
\[
  \Phi(w)=w',\quad \Phi^\ast(\mu') =\mu,\quad \Phi^\ast(p')=p.
\]
We associate to any simply-connected, closed and oriented smooth $6$-manifold $M$ with vanishing odd-degree integer cohomology the system of invariants
\[
   \mathcal S(M):=\left( H^2(M,\ZZ),\mu_M, w_2(M), p_1(M)  \right).
\]
Then we recall
\begin{thm}[Wall, Jupp, \v Zubr]\label{T:Classification}
  Let $M$ and $N$ be compact, simply-connected, oriented, smooth $6$-manifolds with
  $H^{\text{odd}}(M,\ZZ)=H^{\text{odd}}(N,\ZZ)=0$. Then any isomorphism $\Phi\colon\mathcal S(N)\to
  \mathcal S(M)$ is realized by an orientation-preserving diffeomorphism $M\to N$.
\end{thm}

The fact that the equivalence class of the system of invariants $S(M)$ determines the diffeomorphism type of $M$ was shown by Jupp \cite[Theorem 1]{MR0314074}, building upon work of Wall \cite{MR0215313} who proved the spin case. Note that they also allowed nonvanishing (torsion-free) odd-degree cohomology. The fact that every equivalence of systems of invariants is realized by a diffeomorphism was proven by \v Zubr in \cite[Theorem 3]{MR970082}. He also allowed for torsion in the cohomology. See  \cite{MR1365849} for a nice overview on the topic.

Combining this with part (a) of Theorem \ref{thm:gkm6dim} we obtain part (b). The remaining statement in (c), i.e., the fact that the homotopy class on an invariant almost
complex structure $J$ is determined by the signed GKM graph, follows from  \cite[Theorem 9]{MR0215313}, where Wall showed that the homotopy class of $J$ is uniquely determined by the first Chern class $c_1(M,J)$.

\section{Tolman, Woodward, and Eschenburg}\label{sec:TWE}

In \cite{Tolman}, Tolman constructed the first example of a compact,  simply-connected symplectic manifold with an Hamiltonian torus action with finite fixed point set, which does not admit any invariant K\"ahler structure. To obtain her example, she started with two six-dimensional toric symplectic manifolds $M_1$ and $M_2$, restricted the actions to two-dimensional subtori, and glued two open subsets of these $T^2$-manifolds together to obtain her example $M_3$.
\begin{lem}\label{lem:tolman}
Tolman's example satisfies all assumptions of Corollary \ref{cor:mainthm}.
\end{lem}
\begin{proof}
The example is simply-connected by \cite[Lemma 4.1]{Tolman}, because there is a component of the momentum map all of whose critical points have even index. It thus suffices to check that all its isotropy groups are connected, and that the action satisfies the GKM conditions.

The connectedness of the isotropy groups was mentioned in \cite{Tolman}, but we include an argument for completeness. We only need to show that the isotropy groups of $M_1$ and $M_2$ are connected. More precisely, the first of these manifolds is $\CC P^1\times \CC P^2$, with the $T^2$-action
\[
(s,t) \cdot ([x_0:x_1],[y_0:y_1:y_2]) = ([sx_0:x_1],[sy_0:ty_1:y_2]),
\]
which obviously has all isotropy groups connected. The second action can be understood via its momentum image, see also \cite[Section 2]{Woodward}. Consider the six-dimensional toric manifold which has as momentum image a polytope in $(\mft^3)^*\cong \RR^3$ whose projection onto the $xy$-plane is as follows:
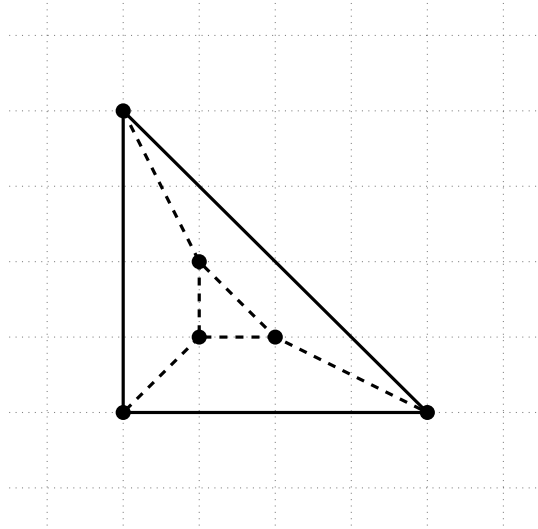
\begin{center}
\begin{tikzpicture}
\draw[step=1, dotted, gray] (-3.5,-4.5) grid (3.5,2.5);

\draw[very thick] (-2,1) -- ++(4,-4) -- ++(-4,0) -- ++(0,4);
\draw[very thick, dashed] (-2,1)--++(1,-2)--++(1,-1)--++(2,-1);
\draw[very thick, dashed] (-1,-1)--++(0,-1)--++(1,0);
\draw[very thick, dashed] (-1,-2)--++(-1,-1);

  \node at (-2,1)[circle,fill,inner sep=2pt]{};

  \node at (2,-3)[circle,fill,inner sep=2pt]{};

  \node at (-2,-3)[circle,fill,inner sep=2pt]{};

  \node at (0,-2)[circle,fill,inner sep=2pt]{};

  \node at (-1,-1)[circle,fill,inner sep=2pt]{};

  \node at (-1,-2)[circle,fill,inner sep=2pt]{};
\end{tikzpicture}
\end{center}
In this and the following pictures of x-rays, the lines, dashed or not, are the images of closures of the
nontrivial orbit type strata.
%
The three-dimensional polytope has the outer triangle at $z=0$ and the inner triangle at $z=1$. One considers the $T^2=T^2\times \{1\}\subset T^3$-subaction. The isotropy groups of the $T^3$-action can be read off as the intersections between the $T^3$-isotropies and $T^2$, and one easily checks that they are all connected: The occurring $T^3$-isotropy groups are the connected subgroups whose Lie algebras are given by all possible intersections of kernels of (one or more) weights at a single fixed point of the action, and the weights are given by the edges in the graph above. For instance, consider the upmost fixed point, and the two edges in direction $(1,-2,1)$ and $(1,-1,0)$. The intersection of their kernels is the subgroup $\{(s,s,s)\}$, whose intersection with $T^2$ is the trivial group, and in particular connected.

The x-ray of the invariant symplectic structure obtained from the gluing process is as follows, see \cite[p.\ 304]{Tolman}:
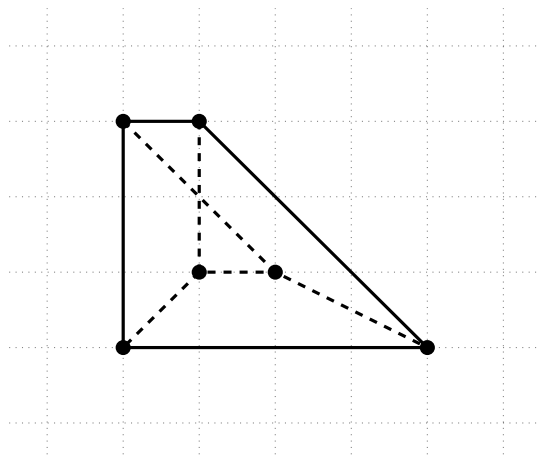
\begin{center}
\begin{tikzpicture}
\draw[step=1, dotted, gray] (-3.5,-4.5) grid (3.5,1.5);

\draw[very thick] (-2,0) -- ++(1,0) -- ++(3,-3) -- ++(-4,0) -- ++(0,3);
\draw[very thick, dashed] (-2,0)--++(2,-2)--++(2,-1);
\draw[very thick, dashed] (-1,0)--++(0,-2)--++(1,0);
\draw[very thick, dashed] (-1,-2)--++(-1,-1);

  \node at (-2,0)[circle,fill,inner sep=2pt]{};

  \node at (-1,0)[circle,fill,inner sep=2pt]{};

  \node at (-2,-3)[circle,fill,inner sep=2pt]{};

  \node at (0,-2)[circle,fill,inner sep=2pt]{};
  \node at (2,-3)[circle,fill,inner sep=2pt]{};
  \node at (-1,-2)[circle,fill,inner sep=2pt]{};
\end{tikzpicture}
\end{center}
This image implies that at every fixed point any two weights of the isotropy representation are linearly independent, i.e., that the action satisfies the GKM conditions.
\end{proof}
The reason why $M_3$ does not admit an invariant K\"ahler structure is that the shape of the momentum image of an invariant symplectic structure would be incompatible with Atiyah's convexity theorem \cite{Atiyah} for orbit closures of the (holomorphic) action of the complexified torus, see \cite[Section 3]{Tolman}.

\begin{rem}
The manifold $M_3$ was constructed by gluing two open $T$-manifolds together. In general it is not
  clear that the resulting manifold does not depend on this diffeomorphism, as the example of
  exotic spheres shows, cf. \cite{MR0082103}. But since the GKM graph of $M_3$ does not depend on
  such a diffeomorphism we obtain from Theorem \ref{thm:gkm6dim} the following
\end{rem}

\begin{cor}
The diffeomorphism type of $M_3$ is unique.
\end{cor}

In \cite{Woodward}, Woodward constructed a very similar example in a different way: he applied $\U(2)$-equivariant symplectic surgery to the full flag manifold $\U(3)/T^3$. The result is a symplectic manifold with very much the same properties as Tolman's example, but with the additional property that the $T^2$-action extends to a multiplicity-free Hamiltonian $\U(2)$-action.

\begin{lem} Woodward's example satisfies all assumptions of Corollary \ref{cor:mainthm}.
\end{lem}
\begin{proof} The example is simply-connected, for the same reason as in Lemma \ref{lem:tolman}. By \cite[Proposition 3.6]{Woodward} the x-ray of his example is the same as that of Tolman's example. Thus, his example satisfies the GKM conditions.

The example is constructed from $\U(3)/T^3$ by symplectic cutting with respect to a local $T^2$-action that commutes with the $\U(2)$-action acting by left multiplication in the upper left block. It follows from a close look at Woodward's construction \cite[p.\ 318]{Woodward} that all isotropy groups of his example are connected. More precisely, as the $T^2$-isotropy groups on $\U(3)/T^3$ are connected, all that is left to check are the newly introduced isotropy groups on the level of the cut. Woodward showed that, in his notation, the isotropy groups occurring in the space $\mu^{-1}(a)\cap Y_+$ are $\{1\}\times \U(1)$ and $\{1\}$, which, after passing to the quotient with respect to $\U(1)_{1,2} = \{(z,z^2)\mid z\in \U(1)\}$, give the isotropy groups $T^2$ and $\U(1)_{1,2}$. Now one observes that all the isotropy groups of the $T^2$-action on the homogeneous spaces $\U(2)/T^2$ and $\U(2)/\U(1)_{1,2}$ by left multiplication are connected. This implies the claim.
\end{proof}

In \cite{1812.09689v1}, we constructed a symplectic structure and an Hamiltonian $T^2$-action on Eschenburg's twisted flag manifold. This manifold can be defined as $M=\SU(3)//T$, where the torus $T=T^2$ acts as
\[(s,t)\cdot A=\begin{pmatrix}
s^2t^2& & \\ & 1 &\\ & & 1
\end{pmatrix}
A
\begin{pmatrix}
\overline{s}& & \\ & \overline{t} &\\ & & \overline{st}
\end{pmatrix}\]
for $s,t\in S^1$. The symplectic form on $M$ is such that the action of $T'=T^2$ induced by

\[(s,t)\cdot A= \begin{pmatrix}
st& & \\ & \overline{s} &\\ & & \overline{t}
\end{pmatrix}A\]
is Hamiltonian. One observes that this action extends in a similar way as Woodward's example: with respect to an appropriately chosen symplectic form the extension of the $T'$-action by left multiplication to $\U(2)\cong {\mathrm{S}}(\U(1)\times \U(2))$ is a multiplicity-free Hamiltonian action, see \cite[Section 3.1]{1812.09689v1}. In \cite{1812.09689v1}, we showed that, up to rescaling of the edges, the x-ray of the action is given by 
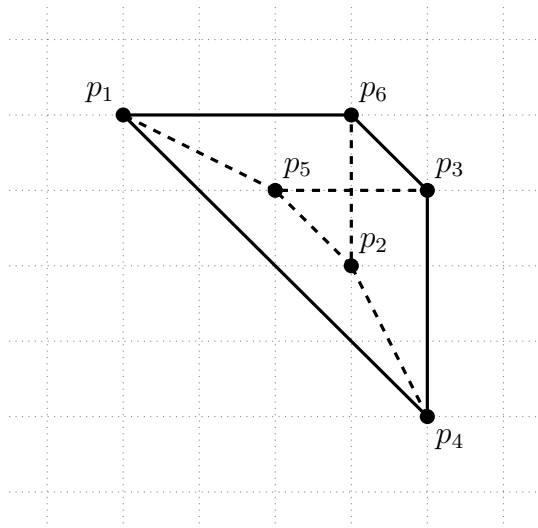
\begin{center}
\begin{tikzpicture}
\draw[step=1, dotted, gray] (-3.5,-4.5) grid (3.5,2.5);

\draw[very thick] (1,1) -- ++(-3,0) -- ++(4,-4) -- ++(0,3) -- ++(-1,1);
\draw[very thick, dashed] (2,-3)--++(-1,+2)--++(0,2)++(0,-2)--++(-1,1)--++(2,0)++(-2,0)--++(-2,1);
  \node at (1,1)[circle,fill,inner sep=2pt]{};
  \node at (1.3,1.3){$p_6$};

  \node at (-2,1)[circle,fill,inner sep=2pt]{};
  \node at (-2.3,1.3){$p_1$};

  \node at (1,-1)[circle,fill,inner sep=2pt]{};
\node at (1.3,-0.7){$p_2$};

  \node at (0,0)[circle,fill,inner sep=2pt]{};
  \node at (0.3,0.3){$p_5$};

  \node at (2,0)[circle,fill,inner sep=2pt]{};
  \node at (2.3,0.3){$p_3$};

  \node at (2,-3)[circle,fill,inner sep=2pt]{};
  \node at (2.3,-3.3){$p_4$};

\end{tikzpicture}
\end{center}
where the $p_i$ are the fixed points.
Note that this picture is uniquely determined by the weights, i.e., the slopes of the edges, and the lengths of the two edges $p_1p_6$ and $p_6p_3$. In \cite{1812.09689v1} we constructed a two-parameter family of appropriate symplectic forms, which realize variations of these two lengths. But note that we did not determine the precise range of occuring ratios between these lengths. We will not need this information for what follows.

\begin{lem}
The $T'$-action on $\SU(3)//T$ satisfies all assumptions of Theorem \ref{thm:gkm6dim}.
\end{lem}
\begin{proof}
Because $\SU(3)$ is simply-connected and $T$ connected, the biquotient $\SU(3)//T$ is simply-connected. The shape of the x-ray described above implies that the action satisfies the GKM conditions. The integer cohomology of $\SU(3)//T$ was computed by Eschenburg in \cite{MR1160094}; in particular it vanishes in odd degrees.

To show that the isotropy groups of this action satisfy the assumptions needed in order for the Chang-Skjelbred lemma to hold true, we only have to observe that the weights of the action, as computed in \cite{1812.09689v1}, are primitive elements in $\ZZ^2\subset \RR^2\cong (\mft^2)^*$, and apply \cite[Lemma 6.1]{MR3456711}. Alternatively, one can compute the isotropy groups directly.
\end{proof}

\begin{thm} \label{thm:TWE} The manifolds from Tolman's and Woodward's examples are both diffeomorphic to the Eschenburg flag manifold $\SU(3)//T$, via diffeomorphisms that respect the homotopy class of a compatible almost complex structure.
\end{thm}
\begin{proof}
We have shown above that all three actions, Tolman's, Woodward's and the example on the Eschenburg flag, satisfy the assumptions of Theorem \ref{thm:gkm6dim} respectively Corollary \ref{cor:mainthm}. Hence their diffeomorphism type is determined by their GKM graphs. We now observe that the x-ray of Tolman's (and Woodward's) example turns (up to rescaling) into that of the Eschenburg example after applying the shear mapping $(x,y)\mapsto (x,x+y)$, followed by a reflection. As in all three examples all occurring weights are primitive elements in $\ZZ_{\mft}^*$, and taking into account Remark \ref{rem:xraygkm}, it follows that the signed GKM graphs of the three examples are isomorphic.

\end{proof}

\begin{rem}
In \cite{MR1160094} Eschenburg computed the cohomology of $M=\SU(3)//T$ via the surjection $H^*(BT)\rightarrow H_T^*(\SU(3))\cong H^*(M)$. We fix the isomorphism $H^*(BT)\cong \mathbb{Z}[X_1,X_2]$ with $X_i$ in degree $2$ induced by the standard basis of $T=T^2$. Eschenburgs methods yield (in a  presentation that is slightly different from \cite{MR1160094}) \[H^*(M)\cong \mathbb{Z}[X_1,X_2]/(X_1^2+3X_1X_2+X_2^2,~X_1^2X_2+X_1X_2^2).\]
The GKM machinery provides a different description of the cohomology ring, which is more complicated but allows for an explicit description of the characteristic classes through the formulas from Proposition \ref{prop:Formeln}. Translating between the two is a rather lengthy but straight forward computation. One obtains that, with respect to the above isomorphism, the Chern classes are
\[
  c_1(M) = 4X_1 +2X_2,\quad c_2(M)=6X_1^2+6X_1X_2,\quad c_3(M)=-6 X_1^2X_2,
\]
the Stiefel-Whitney classes are all zero and the first Pontrjagin class is given by
\[
  p_1(M)  = -8X_1X_2.
\]
Six-dimensional Hamiltonian GKM manifolds with six fixed points were also considered in \cite{Morton}; see Section 4.3 therein for the computation of the Chern classes of this class of examples via GKM theory.

\end{rem}

\begin{rem} Theorem \ref{thm:TWE} is a statement about the diffeomorphism type of Tolman's
and Woodward's example, not the equivariant diffeomorphism type. If one
is interested in equivariant diffeomorphism one first needs to observe
that when one speaks about these examples, one in fact means a whole
variety of Hamiltonian $T^2$-manifolds: As Tolman's construction depends
on the choice of a gluing map, it is a priori not uniquely defined up to
equivariant symplectomorphism, as observed in \cite[Remark 2.7]{1912.02785v2}. Moreover, Woodward considers in his recap of Tolman's example
\cite[Section 2]{Woodward} a variant of her construction, in which he glues
different manifolds. While the x-ray of these examples coincide with
that of Tolman's example, so that Theorem 4.6 is applicable to all of
them, it is not known if they are equivariantly diffeomorphic.

To show the $T^2$-equivariance of the examples one unfortunately cannot
apply the existing theory on complexity one spaces, see e.g. \cite{MR2128388}
as these $T^2$-manifolds are not tall. A possibility how one might prove
$S^1$-equivariant diffeomorphism of these examples, with respect to some
subcircle $S^1\subset T^2$,
is outlined in \cite[Remark 2.7]{1912.02785v2}.
Instead of restricting the action one can also try to extend it: both
the action on Woodward's example and on the Eschenburg flag admit an
extension to a multiplicity-free $\U(2)$-action, so that one might
apply Knop's solution of Delzant's conjecture on multiplicity-free
actions \cite{MR2748401}.
Note that it
is not known if Tolman's original example(s) and Woodward's variant
admit an extension to $\U(2)$.
This approach is the topic of an ongoing
master thesis of Nikolas Wardenski.
\end{rem}

It is known that the Eschenburg flag admits a K\"ahler structure; this is implicit in work of Eschenburg \cite[Theorem 2]{Eschenburg2} and Escher--Ziller \cite{EscherZiller}, and explicit in \cite[Section 4]{1812.09689v1}. This K\"ahler structure can not be $T^2$-invariant by the work of Tolman \cite{Tolman}, which applies equally to the Eschenburg flag as the arguments operate on the level of the x-ray.  We will now investigate the occurring forms more closely and show that not only does there exist some K\"ahler structure, but actually the $T^2$-invariant symplectic form which we constructed on the Eschenburg flag in \cite{1812.09689v1} is a K\"ahler form itself (with the property that a compatible complex structure can not be $T^2$-invariant).

In order to see this we remind the reader of how the $T^2$-invariant symplectic form $\omega_H$ and a Kähler form $\omega_K$ on $\SU(3)//T$ can be defined  (see \cite[p.\
  77f]{Voisin} and \cite[Section 3.1]{1812.09689v1}) and then show that the two are symplectomorphic:
  \[
  \SU(3)//T = \SU(3)\times_{\U(2)} \U(2)/T^2
  \] is the projectivization $\mathbb P(E)$ of the rank 2 complex vector
     bundle
     \[
     E = \SU(3)\times_{\U(2)} \CC^2 \to \CC P^{2},
     \] cf.\ \cite[Proposition 4.3]{1812.09689v1}. Every (topological) vector
     bundle over $\CC P^2$ admits a holomorphic structure \cite[p.\ 63]{MR2815674} and therefore
     $\mathbb P(E)$ is a complex manifold such that $\pi\colon \mathbb P(E) \to \CC P^2$ is
     a holomorphic map. Observe that we have a canonical isomorphism of fibers ${\mathbb P}(E_p)\cong \U(2)/T^2$ up to elements of $\U(2)$.

     We fix the Hermitian metric $h$ on $E \to \CC P^2$ given by the standard Hermitian metric on each fiber $\CC^2$. Then we obtain a Hermitian
     metric on $\mathcal O_{\mathbb P(E)}(1)$, the dual of the tautological bundle over every fiber of
     $\mathbb P(E)$. Denote by $\omega_F$ the Chern curvature of the induced Hermitian metric on
     $\mathcal O_{\mathbb P(E)}(1)$. It restricts on every fiber of $\mathbb P(E)$ to the same form on ${\mathbb P}(E_p)\cong \CC P^1\cong \U(2)/T^2$, namely the ($\U(2)$-invariant) Fubini-Study form on $\CC P^1$, where the isomorphism ${\mathbb P}(E_p)\cong \CC P^1$ is induced by an isometry $E_p\cong \CC^2$. Observe that by our choice of metric the induced isomorphism ${\mathbb P}(E_p)\cong \U(2)/T^2$ is the canonical one above. Letting $\omega_B$ be the Fubini-Study form on $\CC P^2$, then
     for $C>0$ big enough the $2$-form
  \[
    \omega_K:=\omega_F + C \cdot \pi^{\ast}(\omega_B)
  \]
  is a Kähler form on $\mathbb P(E)$ (see \cite[Proposition 3.18]{Voisin}). For the construction of $\omega_H$
  (see \cite[Section 3.1]{1812.09689v1}) we can average $\omega_F$ over $T^2$
  \[
    \widetilde \omega_F = \int_{T^2} t^\ast(\omega_F)  \,dt
  \]
  and set
  \[
    \omega_H:= \widetilde\omega_F  + C \cdot \pi^{\ast}(\omega_B)
  \]
  which is still a symplectic form (after possibly replacing $C$ by a bigger constant) because $\omega_F$ and $\widetilde \omega_F$ restrict to the same form on fibers (c.f.\
  \cite[Theorem 2.1]{1812.09689v1}). The form $\pi^*(\omega_B)$ is $T^2$-invariant by the construction of the action on $E$ so $\omega_H$ is $T^2$-invariant. Since averaging a closed form over a compact group does not change the deRham class, we have $\omega_F=\widetilde \omega_F+d\eta$ for some $1$-form $\eta$. Note that $d\eta$ restricts to $0$ on fibers as the restrictions of $\omega_F$ and $\widetilde \omega_F$ agree. In particular the restrictions of $\omega_F+td\eta$ agree and are symplectic for any $t\in\mathbb{R}$. By possibly replacing $C$ with an even bigger constant we achieve that $\omega_t=\omega_F+td\eta+C\cdot\pi^*(\omega_B)$ is symplectic for all $t\in [0,1]$. Thus $\omega_H$ and $\omega_K$ are joined by a path through symplectic forms in the same deRham class and hence are symplectomorphic by Moser's trick.

We sum up the above discussion in the theorem below. In a previous version of this article we only stated that one could pick $\omega_H$ and $\omega_K$ from the same deRham class. We are very grateful to Nicholas Lindsay and Dmitri Panov for pointing us toward the stronger statement below and providing helpful comments regarding its proof.
\begin{thm}\label{thm:Kahler}
On the manifold(s) discussed in this section there is a $T^2$-invariant symplectic form (the action coming from the one on the Eschenburg flag) which admits a compatible complex structure but no $T^2$-invariant compatible complex structure.
\end{thm}

\bibliographystyle{acm}
\bibliography{GKMdimension6-revision.bbl}
\end{document}